\documentclass[a4paper,11pt]{amsart}
\usepackage{amsmath,amssymb,amsthm, amsfonts,color}
\usepackage{hyperref}
\usepackage{graphicx}
\usepackage{url}
\usepackage{dsfont} 

\newtheorem{thm}{Theorem}[section]
\newtheorem*{thm*}{Theorem}
\newtheorem{cor}[thm]{Corollary}
\newtheorem{lem}[thm]{Lemma}
\newtheorem{prop}[thm]{Proposition}
\newtheorem{obs}[thm]{Example}

\theoremstyle{definition}
\newtheorem{rem}[thm]{Remark}

\theoremstyle{definition}

\theoremstyle{definition}

\numberwithin{equation}{section}

\newcommand{\la}{\langle}
\newcommand{\ra}{\rangle}
\newcommand{\N}{\mathbb{N}}
\newcommand{\Z}{\mathbb{Z}}
\newcommand{\norm}[1]{\left\lVert#1\right\rVert}
\newcommand{\abs}[1]{|#1|}

\newcommand{\mch}{\mathcal{H}}
\newcommand{\mcb}{\mathcal{B}}
\def\({\left(}
\def\){\right)}

\allowdisplaybreaks

\title{The Dual Kaczmarz Algorithm}

\author[A. Aboud]{Anna Aboud}
\author[E. Curl]{Emelie Curl}
\author[S. N. Harding]{Steven N. Harding}
\author[M. Vaughan]{Mary Vaughan}
\author[E. S. Weber]{Eric S. Weber}
\address{Department of Mathematics\\
Iowa State University\\
396 Carver Hall, Ames\\
IA 50011, United States of America}
\email[A. Aboud]{acseitz@iastate.edu}
\email[E. Curl]{ecurl@iastate.edu}
\email[S. N. Harding]{sharding@iastate.edu} 
\email[M. Vaughan]{maryo@iastate.edu}
\email[E. S. Weber]{esweber@iastate.edu}



\begin{document}
\maketitle

\begin{abstract} 
The Kaczmarz algorithm is an iterative method for solving a system of linear equations.  It can be extended so as to reconstruct a vector $x$ in a (separable) Hilbert space from the inner-products $\{\langle x, \phi_{n} \rangle\}$. The Kaczmarz algorithms defines a sequence of approximations from the sequence $\{\langle x, \phi_{n} \rangle\}$; these approximations only converge to $x$ when $\{\phi_{n}\}$ is \emph{effective}. We dualize the Kaczmarz algorithm so that $x$ can be obtained from $\{\langle x, \phi_{n} \rangle\}$ by using a second sequence $\{\psi_{n}\}$ in the reconstruction.  This allows for the recovery of $x$ even when the sequence $\{\phi_{n}\}$ is not effective; in particular, our dualization yields a reconstruction when the sequence $\{\phi_{n}\}$ is \emph{almost effective}.  
We also obtain some partial results characterizing when the sequence of approximations from $\{\langle x, \phi_{n} \rangle\}$ using $\{\psi_{n}\}$ converges to $x$, in which case $\{(\phi_n, \psi_n)\}$ is called an \emph{effective pair}.   
  
\end{abstract}

\section{Introduction}\label{sintro}  

A question of consistent interest in recent years is the conditions under which one can reconstruct a vector in a Hilbert space $\mch$ given its inner products with some sequence of vectors $\{e_n\}_{n=0}^{\infty}$. Sequences of vectors $\{e_n\}_{n=0}^{\infty}$ which yield non-uniqueness of representations and robustness to perturbations have motivated advances in the area of frame theory \cite{Caz-Chr,Chr1,Chr2}. There is now renewed interest in iterative reconstructions, e.g.~phase retrieval \cite{Jeong,Tan}, optimization \cite{N-all}, learning theory \cite{Kwa-Myc}, and computerized tomography \cite{Nbook}, to name a few. 
All of these areas make use of the reconstruction method which we now introduce.

In 1937, Stefan Kaczmarz introduced an iterative process of solving linear systems which we now call the \textit{Kaczmarz algorithm}. Given a linearly dense sequence of unit vectors $\{e_n\}_{n=0}^{\infty}$ in $\mch$ and $x \in \mch$, we define a sequence of approximations $\{x_n\}_{n=0}^{\infty}$ by
\begin{equation}\label{defn:ka}
\begin{aligned}
		x_0 &= \la x, e_0 \ra e_0, \\
		x_n &= x_{n-1} + \la x- x_{n-1}, e_n \ra e_n, \quad n \geq 1.
\end{aligned}
\end{equation}
We say that $\{e_n\}_{n=0}^{\infty}$ is \textit{effective} if $\norm{x_n - x} \to 0$ for every $x \in \mathcal{H}$. Kaczmarz showed in \cite{Kpaper} that any periodic, linearly dense sequence of unit vectors $\{e_n\}_{n = 1}^{\infty}$ in a finite-dimensional Hilbert space is effective.

In \cite{Kwa-Myc}, Kwapie\'{n} and Mycielski made progress in the infinite-dimensional Hilbert space setting by utilizing the auxiliary sequence $\{h_n\}_{n=0}^{\infty}$ defined recursively as 
\begin{equation}\label{defn:aux}
\begin{aligned}
h_0 &= e_0, \\
h_n &= e_n - \sum_{k=0}^{n-1} \la e_n, e_k \ra h_k, \quad n \geq 1.
\end{aligned}
\end{equation}
They showed that $\{e_n\}_{n=0}^{\infty}$ is effective if and only if $\{h_n\}_{n=0}^{\infty}$ is a Parseval frame, namely that, for every $x \in \mathcal{H}$,
\[\|x\|^2 = \sum_{n = 0}^{\infty}|\langle x,h_n\rangle|^2.\]
In \cite{Hall-Szw}, Haller and Szwarc approached the characterization from a different perspective, using the matrix of inner products
\begin{equation}\label{I+M}
I+M = 
\begin{pmatrix}
1 					& 0 				& 0 				& \cdots \\
\la e_1, e_0 \ra 	& 1					&0					& \cdots \\
\la e_2, e_0 \ra 	& \la e_2, e_1 \ra	&1					& \cdots \\
\la e_3, e_0 \ra 	& \la e_3, e_1 \ra	&\la e_3, e_2 \ra 	& \cdots \\
\vdots				& \vdots			& \vdots			& \ddots \\
\end{pmatrix}.
\end{equation}
They showed that $\{e_n\}_{n=0}^{\infty}$ is effective if and only if the matrix $U$ is a partial isometry where $I + U$ is the algebraic inverse of $I + M$.

In the more general setting of a Banach space $X$, Kwapie\'{n} and Mycielski in \cite{Kwa-Myc} reinterpreted the Kaczmarz algorithm as the iterative process associated to $\{(e_n,f_n)\}_{n=0}^{\infty} \subset X \times X^*$ with ${f_n(e_n) = 1}$ and
	\begin{align*}
		x_0 &= 0, \\
		x_n &= x_{n-1} + f_n(x - x_{n-1}) e_n, \quad n \geq 1.
	\end{align*}
In this context, the sequence $\{(e_n,f_n)\}_{n=0}^{\infty}$ is said to be effective if $\norm{x - x_n}_{X} \rightarrow 0$ for every $x \in X$. 
Although Kwapie\'n and Mycielski did not produce a characterization of effective sequences in Banach spaces, they did attain approximation results by thresholding the linear functionals.

With this extension to Banach spaces, we find applications of the Kaczmarz algorithm in learning theory. In \cite{Kwa-Myc}, Kwapie\'{n} and Mycielski describe a model concerning the space 
\[X = \{x \in C(\mathbb{R}) : x(t + T) = x(t)\}.\]
They ask when there exists some $h \in X$ with $h(0) = 1$ and some sequence $\{t_n\}_{n = 0}^{\infty}$ such that every $x \in X$ can be recognized (uniformly approximated) from the sequence $\{x_n\}_{n=0}^{\infty}$ generated by the learning process
\begin{equation}\label{defn:ka in Banach}
\begin{aligned}
x_0(t) &= 0, \\
x_n(t) &= x_{n-1}(t) + (x(t_n) - x_{n-1}(t_n))h(t - t_n), \quad n \geq 1.
\end{aligned}
\end{equation}
By letting $e_n = h(t - t_n)$ and $f_n(x) = x(t_n)$, the question is resolved by determining when the sequence $\{(e_n,f_n)\}_{n=0}^{\infty}$ is effective.

\subsection{Introduction to frame theory}

The results in this paper implement tools and concepts from frame theory. We will introduce needed terminology and definitions, referring the reader to \cite{Cbook} for a more thorough discussion. 

In a Hilbert space $\mathcal{H}$, a \textit{frame} is a sequence of vectors $\{f_n\}_{n=0}^{\infty}$ for which there exist positive constants $A$ and $B$ such that
\begin{equation}\label{defn:frame}
A\|x\|^2 \leq \sum_{n = 0}^{\infty}|\langle x,f_n\rangle|^2 \leq B\|x\|^2 \quad\hbox{for all}~x \in \mathcal{H}.
\end{equation}
A frame is \textit{tight} if $A = B$ and \textit{Parseval} if $A = B = 1$. The sequence $\{f_n\}_{n=0}^{\infty}$ (not necessarily a frame) is \textit{Bessel} if the positive constant $B$ exists in Equation \eqref{defn:frame}. A \textit{Riesz basis} is a frame which ceases to be a frame if any of its vectors are removed.

Supposing $\{f_n\}_{n=0}^{\infty}$ and $\{g_n\}_{n=0}^{\infty}$ are frames, if for every $x \in \mathcal{H}$ 
\[x = \sum_{n=0}^{\infty}\langle x,f_n\rangle g_n,\]
then we say that
$\{g_n\}_{n=0}^{\infty}$ is a \textit{dual frame} for $\{f_n\}_{n=0}^{\infty}$. The dual frame condition is always symmetric, that is, if $\{g_n\}_{n=0}^{\infty}$ is a dual frame for $\{f_n\}_{n=0}^{\infty}$, then $\{f_n\}_{n=0}^{\infty}$ is a dual frame for $\{g_n\}_{n=0}^{\infty}$.

The frame condition in Equation \eqref{defn:frame} guarantees the existence of dual frames. In fact, given that $\{f_n\}_{n=0}^{\infty}$ is a frame, the operator $S : \mathcal{H} \rightarrow \mathcal{H}$ defined by $Sx = \sum_{n=0}^{\infty}\langle x,f_n\rangle f_n$ is bounded, positive, and invertible.
One can check that $\{S^{-1}f_n\}_{n=0}^{\infty}$ is a dual frame for $\{f_n\}_{n=0}^{\infty}$; indeed, it is referred to as the \textit{canonical dual frame}.
In general, dual frames are not unique. However, if $\{f_n\}_{n=0}^{\infty}$ is a Riesz basis, then its dual frame $\{g_n\}_{n=0}^{\infty}$, which is also a Riesz basis, is unique. Moreover, it can be shown that $\la f_m, g_n \ra = \delta_{m,n}$ for all $m, n \in \N_0$, leading us to refer to  $\{f_n\}_{n=0}^{\infty}$ and $\{g_n\}_{n=0}^{\infty}$ as \textit{biorthogonal Riesz bases}.

A frame provides reconstruction of any $x \in \mathcal{H}$ from inner products with the frame elements while also allowing for redundancy among those elements. 
This procedure is captured by the analysis and synthesis operators. The \textit{analysis operator} associated with the sequence $\{f_n\}_{n=0}^{\infty}$ is the map $\Theta_{f} : \mathcal{H} \rightarrow c(\N_0)$ given by
\begin{equation}\label{defn:analysis op}
\Theta_f(x) = \{\langle x,f_n\rangle\}_{n=0}^{\infty},
\end{equation}
where $c(\N_0)$ is the space of sequences on $\N_0$. 
When $\{f_n\}_{n=0}^{\infty}$ is Bessel, the operator $\Theta_f$ is bounded from $\mch$ into $\ell^2(\N_0)$.
However, this condition is not always assumed. Let $l(\N_0)$ denote the subspace of sequences $\{c_n\}_{n=0}^{\infty} \in c(\N_0)$ for which $\sum_{n=0}^{\infty} c_nf_n$ converges. The \textit{synthesis operator} associated with $\{f_n\}_{n=0}^{\infty}$ is the map $\Theta_f^* : l(\N_0) \rightarrow \mathcal{H}$ given by
\begin{equation}\label{defn:synthesis op}
\Theta_f^*(\{c_n\}_{n=0}^{\infty}) = \sum_{n = 0}^{\infty} c_nf_n.
\end{equation}
When $\{f_n\}_{n=0}^{\infty}$ is Bessel, we may replace $l(\N_0)$ by $\ell^2(\N_0)$ and then, as the notation suggests, the synthesis operator is the Hilbert space adjoint of the analysis operator. The \textit{frame operator} is $\Theta_f^*\Theta_f$ and is the identity operator if and only if $\{f_n\}_{n=0}^{\infty}$ is a Parseval frame.

We will consider the mixed Grammian matrix $\Theta_{\phi} \Theta_{\psi}^*$ for sequences $\{\phi_n\}_{n=0}^{\infty}$ and $\{\psi_n\}_{n=0}^{\infty}$ given by 
\begin{equation}\label{defn:grammian}
(\Theta_{\phi} \Theta_{\psi}^*)_{m,n} = \langle \psi_m, \phi_n \rangle.  
\end{equation}
This matrix in general does not define a bounded operator on $\ell^2(\N_0)$, so we need to take care when referring to its positivity. For our purposes, we say that an infinite matrix $T$ is positive if every principle submatrix is positive. That is, for all $n \in \N_0$, the operator
\[T_n = \begin{pmatrix}
t_{00} & t_{01} & \dots & t_{0n} & 0 & \dots \\
t_{10} & t_{11} & \dots & t_{1n} & 0 & \dots\\
\vdots & \vdots & \ddots & \vdots & \vdots \\
t_{n0} & t_{n1} & \dots & t_{nn}& 0 &   \\
0 & 0 & \dots & 0 & 0& \ddots \\
\vdots & \vdots & & &\ddots & \ddots
\end{pmatrix}\]
satisfies $\la T_n u, u \ra \geq 0$ for every $u \in \ell^2(\N_0)$.

Throughout this paper, we adopt the convention that indexing starts at $0$, unless otherwise stated, so indexing notation will be dropped when understood. 

\section{The dual Kaczmarz algorithm}

While effective sequences are useful in vector recovery, they need not retain their efficacy when subject to perturbation. This was shown by Czaja and Tanis in \cite{Cza-Tan} when they proved that a Riesz basis which is not an orthonormal basis cannot be effective. The counterexample then follows directly from a classic result of Paley and Wiener (see \cite{P-W}), namely that a sufficiently small perturbation of an orthonormal basis, which is necessarily effective, may produce a Riesz basis which is not an orthonormal basis---and hence no longer effective. With the intention of obtaining a better tolerance to perturbation, we introduce a variation on the Kaczmarz algorithm where two sequences work together to achieve reconstruction, in analogy to dual frames.

Let $\{\phi_n\}$ and $\{\psi_n\}$ be two linearly dense sequences in $\mch$ such that $\la \phi_n, \psi_n \ra = 1$. Given $x \in \mch$, we define the \textit{dual Kaczmarz algorithm} applied to $x$ by
\begin{equation}\label{defn:dual ka}
\begin{aligned}
x_0 &= \la x, \phi_0 \ra \psi_0, \\
x_n &= x_{n-1} + \la x-x_{n-1}, \phi_n \ra \psi_n, \quad n \geq 1.
\end{aligned}
\end{equation}
If $\norm{x-x_n} \to 0$ for all $x \in \mch$, then we say that $\{\phi_n\}$ and $\{\psi_n\}$ form an \textit{effective pair}. As will be demonstrated in Example \ref{obs:symmetry}, efficacy need not be preserved when $\phi_n$ and $\psi_n$ are interchanged in the algorithm. Hence, we will call the first sequence $\{\phi_n\}$ the analysis sequence and the second sequence $\{\psi_n\}$ the synthesis sequence, representing the ordering by $\{(\phi_n,\psi_n)\}$. If both $\{\(\phi_n, \psi_n\)\}$ and $\{\(\psi_n, \phi_n\)\}$ are effective pairs, we say that the sequences form a \textit{symmetric effective pair}. 
We note that this is distinct from the dual frame condition which is always symmetric.

Similarly to Kwapie\'{n} and Mycielski in \cite{Kwa-Myc}, we recursively define the auxiliary sequence $\{g_n\}$ for a pair $\{(\phi_n,\psi_n)\}$ as follows: 
\begin{equation}\label{defn:dual aux}
\begin{aligned}
g_0 &= \phi_0, \\
g_n &= \displaystyle{\phi_n - \sum_{k=0}^{n-1} \la \phi_n, \psi_k \ra g_k, \quad n \geq 1}.
\end{aligned}
\end{equation}
It is an inductive argument to show
\begin{equation}\label{xn approx}
x_n = \sum_{k=0}^n \la x, g_k \ra \psi_k, \quad n \geq 0.
\end{equation}
Consequently, if $\{(\phi_n,\psi_n)\}$ is an effective pair, then we have the reconstruction
\begin{equation}\label{pair recon}
x = \sum_{k=0}^\infty \la x, g_k \ra \psi_k.
\end{equation}
We likewise form an auxiliary sequence $\{\tilde{g}_n\}$ for $\{(\psi_n,\phi_n)\}$ by
\begin{equation}\label{defn: dual aux2}
\begin{aligned}
\tilde{g}_0 &= \psi_0 \\
\tilde{g}_n &= \displaystyle{\psi_n - \sum_{k=0}^{n-1} \la \psi_n, \phi_k \ra \tilde{g}_k, \quad n \geq 1}.
\end{aligned}
\end{equation}

\begin{rem}
Our notation $g_n, \tilde{g}_n$ suggests duality in the context of frames. Although this is sometimes the case, we also have examples where $\{(\phi_n, \psi_n)\}$ is an effective pair and one or both of $\{g_n\}$ and $\{\tilde{g}_n\}$ fail to be frames. This is shown in Example \ref{obs:not dual frames}.
\end{rem}

\begin{rem}
As an effective sequence forms an effective pair with itself, it is natural to ask whether the two sequences in an effective pair are independently effective. Appealing to Schauder bases which are not Riesz bases, we find many examples for which this is not necessarily the case. See Example~\ref{obs:not dual frames} for more details. 
\end{rem}

\begin{rem}
There are many more effective pairs than there are effective sequences. Indeed, Corollary \ref{cor:T pair to seq} will demonstrate that any effective sequence and invertible operator can generate an effective pair. 
\end{rem}

\begin{rem}
The concept of an effective pair translates more naturally to the context of a Banach space than that of an effective sequence. 
Although Kwapie\'n and Mycielski in \cite{Kwa-Myc} defined 
an effective sequence in a Banach space, $X$, their definition \eqref{defn:ka in Banach} relies upon two different sequences---one in $X$ and one in $X^*$---which is equivalent to our effective pair definition from \eqref{defn:dual ka}. As observed in Example~\ref{obs:not dual frames}, a Schauder basis and its biorthogonal dual are an effective pair. 
\end{rem}

\begin{thm}\label{thm:T to pair}
Let $T \in \mathcal{B}(\mch)$ be invertible. Then 
$\{\(\phi_n, \psi_n\)\}$ is an effective pair if and only if $\{\(T\phi_n, (T^{-1})^*\psi_n\)\}$ is an effective pair.
\end{thm}

\begin{proof}
Suppose that $\{(\phi_n,\psi_n)\}$ is an effective pair. Let $x \in \mch$, and attain the sequence of approximations $\{y_n\}$ using Equation \eqref{defn:dual ka} applied to $T^*x$.
Since $\{(\phi_n,\psi_n)\}$ is an effective pair, we know $\norm{T^*x - y_n} \to 0$. Next, define the sequence $\{x_n\}$ via the dual Kaczmarz algorithm applied to $x$ using $\{(T\phi_n, (T^{-1})^*\psi_n)\}$, i.e.
\begin{align*}
x_0 &= \la x, T\phi_0 \ra (T^{-1})^*\psi_0, \\
x_n &= x_{n-1} + \la x-x_{n-1}, T\phi_n \ra (T^{-1})^*\psi_n, \quad n \geq 1.
\end{align*}
Observe that
\[x_0 = \la x, T\phi_0 \ra (T^{-1})^*\psi_0
    = (T^{-1})^*\( \la T^*x, \phi_0 \ra \psi_0 \)\\
    = (T^{-1})^* y_0.\]
Assume inductively that $x_{n-1} = (T^{-1})^* y_{n-1}$. Then
\begin{align*}
x_n &= (T^{-1})^*y_{n-1} + \la T^*x-T^*x_{n-1}, \phi_n \ra (T^{-1})^*\psi_n\\
    &= (T^{-1})^*\( y_{n-1} + \la T^*x - y_{n-1}, \phi_n \ra \psi_n\) \\
    &= (T^{-1})^* y_n.
\end{align*}
Therefore, $x_n = (T^{-1})^* y_n$ for all $n\in \N_0$. As $T^{-1}$ is bounded, we have
\begin{align*}
\norm{x-x_n} &= \norm{(T^{-1})^*T^{*}(x - x_n)}\\
&\leq \norm{(T^{-1})^*} \norm{T^{*}x - T^{*}x_n} = \norm{(T^{-1})^*} \norm{T^*x - y_n} \to 0,
\end{align*}
so that $\{(T\phi_n,(T^{-1})^*\psi_n)\}$ is an effective pair. Conversely, suppose that $\{(T\phi_n,(T^{-1})^*\psi_n)\}$ is an effective pair. Let $S = T^{-1}$. From the above argument, it follows that
\[\{(ST\phi_n,(S^{-1})^*(T^{-1})^*\psi_n)\} = \{(\phi_n,\psi_n)\}\] 
is an effective pair.
\end{proof}

\begin{cor}\label{cor:T pair to seq}
Let $T \in \mathcal{B}(\mch)$ be invertible. A linearly dense sequence $\{e_n\}$ of unit vectors is effective if and only if $\{(Te_n,(T^{-1})^*e_n)\}$ is a symmetric effective pair. 
\end{cor} 
\begin{proof}
It is clear that $\{e_n\}$ is effective if and only if $\{\(e_n, e_n\)\}$ is an effective pair. Applying Theorem \ref{thm:T to pair} with $T$ and $(T^{-1})^*$, we conclude that $\{\(Te_n, (T^{-1})^* e_n\)\}$ and $\{\((T^{-1})^*e_n, Te_n\)\}$ are both effective pairs if and only if $\{\(e_n, e_n\)\}$ is an effective pair. 
\end{proof}

It would be very advantageous to achieve a characterization of all effective pairs. Most of our results in this direction, however, depend upon the existence of a certain operator $T$ satisfying $\psi_n = T\phi_n$. With this in mind, we next present a string of lemmas that are tied to this condition, each imposing increasingly stringent hypotheses on $T$. It is assumed in every lemma that $\{\phi_n\}$ and $\{\psi_n\}$ are linearly dense in $\mch$ and that $\{g_n\}$ and $\{\tilde{g}_n\}$ are constructed according to Equations \eqref{defn:dual aux} and \eqref{defn: dual aux2}. In these lemmas, as well as in the rest of the paper, we will reference $T^{\frac{1}{2}}$ as the positive square root of $T$, when defined. 
\begin{lem}\label{lem:A1}
If $T \in \mcb(\mch)$ is such that $T g_n = \tilde{g}_n$ and
\begin{equation}\label{eq:phi psi sym}
\la \phi_n, \psi_k \ra = \la \psi_n, \phi_k \ra \ \text{for all} \ n, k \in \N_0,
\end{equation}
then $T\phi_n = \psi_n$ for all $n \in \N_0$.
\end{lem}
\begin{proof}
This is clear for $n = 0$,
\begin{align*}
T\phi_0 = Tg_0 = \tilde{g}_0 = \psi_0.
\end{align*}
For $n \geq 1$, observe that
\[\psi_n - \sum_{k=0}^{n-1}\la \psi_n, \phi_k \ra \tilde{g}_k =\tilde{g}_n = Tg_n = T\phi_n - \sum_{k=0}^{n-1} \la \phi_n, \psi_k \ra Tg_k = T\phi_n - \sum_{k=0}^{n-1} \la \psi_n, \phi_k \ra \tilde{g}_k, \]
so $T \phi_n = \psi_n$, as desired.
\end{proof}

\begin{lem}\label{lem:A2}
Suppose $T \in \mcb(\mch)$ is positive and such that $T\phi_n = \psi_n$ for all $n \in \N_0$. If we define $\{e_n\}$ by
\begin{equation}\label{defn:en}
e_n = T^{\frac{1}{2}}\phi_n,
\end{equation}
then $h_n = T^{\frac{1}{2}}g_n$, where $\{h_n\}$ is the auxiliary sequence to $\{e_n\}$ as in Equation \eqref{defn:aux}. 
\end{lem}

\begin{proof}
First note that 
\begin{equation} \label{eq:en sym}
\la e_m, e_n \ra 
	= \la T^{\frac{1}{2}}\phi_m, T^{\frac{1}{2}}\phi_n \ra 
    = \la \phi_m, T \phi_n \ra
    = \la \phi_m, \psi_n \ra 
\end{equation}
for all $m,n$. Observe that
\[h_0 = e_0 = T^{\frac{1}{2}}\phi_0 = T^{\frac{1}{2}} g_0.\]
Assume inductively that $h_k = T^{\frac{1}{2}}g_k$ for all $0 \leq k <n $. It follows that
\[h_{n} = e_n - \sum_{k=0}^{n-1} \la \phi_{n}, \psi_k \ra h_k = T^{\frac{1}{2}}\( \phi_{n} - \sum_{k=0}^{n-1} \la \phi_{n}, \psi_k \ra  g_k \) = T^{\frac{1}{2}} g_{n} \]
which concludes the induction.
\end{proof}

\begin{rem}
Note that in Lemma \ref{lem:A2} and hereafter the sequence $e_n = T^{\frac{1}{2}}\phi_n$ is not necessarily a sequence of unit vectors. 
\end{rem}

\begin{lem}\label{lem:A3}
Let $T \in \mathcal{B}(\mch)$ be positive and invertible. If $T \phi_n = \psi_n$ for all $n \in \N_0$, then $Tg_n = \tilde{g}_n$ for all $n \in \N_0$.
\end{lem}

\begin{proof}
This is clear for $n=0$,
\[Tg_0 = T \phi_0 = \psi_0 = \tilde{g}_0.\]
Assume inductively that $Tg_k = \tilde{g}_k$ for all $0 \leq k < n$. Observe that
\[T g_n = T\phi_n - \sum_{k=0}^{n-1} \la \phi_n, \psi_k \ra T g_k =  \psi_n - \sum_{k=0}^{n-1} \la T\phi_n, T^{-1}\psi_k \ra \tilde{g}_k = \psi_n - \sum_{k=0}^{n-1} \la \psi_n, \phi_k \ra \tilde{g}_k = \tilde{g}_n\]
Thus, the statement holds for all $n \in \N_0$. 
\end{proof}

\subsection{Towards a characterization of symmetric effective pairs}

After the model of Haller and Szwarc in \cite{Hall-Szw}, we seek necessary and sufficient conditions for a pair of sequences to be an effective pair. As previously mentioned, most of our results in this area depend upon a positive, invertible operator relating $\{\phi_n\}$ and $\{\psi_n\}$. In finite dimensions, we attain such an operator by exploiting the analysis and synthesis operators associated with the given sequences $\{\phi_n\}$ and $\{\psi_n\}$, as seen in Equations \eqref{defn:analysis op} and \eqref{defn:synthesis op}. In infinite dimensions, however, the situation becomes more complex as we are forced to impose various conditions to ensure the existence of such an operator $T \in \mcb(\mch)$.

Corresponding to the pair $\{(\phi_n,\psi_n)\}$, we define 
\begin{equation}\label{I+N}
I+N = \begin{pmatrix}
1 & 0 & 0 & \cdots \\
\la \phi_1,\psi_0\ra & 1 & 0 & \cdots \\
\la \phi_2,\psi_0\ra & \la \phi_2, \psi_1 \ra & 1 & \cdots \\
\la \phi_3,\psi_0\ra & \la \phi_3, \psi_1 \ra & \la \phi_3, \psi_2 \ra & \cdots \\
\vdots & \vdots & \vdots &\ddots \\
\end{pmatrix}
\end{equation}
and $I+V$ as the algebraic inverse of $I+N$. To be sure, this is equivalent to 
\begin{equation} \label{eq:N and V}
(I+V)(I+N) = (I+N)(I+V) = I.
\end{equation}

\begin{prop}\label{prop:V implies sym ep}
Suppose that $\{\phi_n\}$ and $\{\psi_n\}$ are linearly dense sequences in $\mch$ satisfying ${\la \phi_n, \psi_n \ra = 1}$ for all $n \in \N_0$. Furthermore, suppose that there exists a positive, invertible $T \in \mcb(\mch)$ such that $\psi_n = T \phi_n$ for all $n \in \N_0$ and that $V$ is a partial isometry. Then $\{ (\phi_n, \psi_n)\}$ is a symmetric effective pair. 
\end{prop}
\begin{proof}
Define $\{e_n\}$ by Equation \eqref{defn:en}, so that
$\phi_n = T^{-\frac{1}{2}}e_n$ and $\psi_n = T^{\frac{1}{2}}e_n$. Since Equation \eqref{eq:en sym} holds, we have that $M = N$, where $M$ and $N$ are as in Equations \eqref{I+M} and \eqref{I+N}, respectively. 
Consequently, $U = V$, and $\{e_n\}$ is effective by \cite[Theorem~1]{Hall-Szw}. 
By Corollary \ref{cor:T pair to seq}, we conclude that $\{(\phi_n, \psi_n)\}$ is a symmetric effective pair.  
\end{proof}

\begin{prop}\label{prop:dual implies sym ep}
Suppose that $\{\phi_n\}$ and $\{\psi_n\}$ are linearly dense sequences in $\mch$, whose respective auxiliary sequences are $\{g_n\}$ and $\{\tilde{g}_n\}$, as in Equations \eqref{defn:dual aux} and \eqref{defn: dual aux2}. Suppose that ${\la \phi_n, \psi_n \ra = 1}$ for all $n \in \N_0$, that $\la \phi_n, \psi_k \ra = \la \psi_n, \phi_k \ra$ for all $n, k \in \N_0$, and that $\{g_n\}$ and $\{\tilde{g}_n\}$ are canonical dual frames. Then $\{ (\phi_n, \psi_n)\}$ is a symmetric effective pair. 
\end{prop}

\begin{proof}
Since $\{g_n\}$ and $\{\tilde{g}_n\}$ are canonical dual frames, we write $\tilde{g}_n= Tg_n$ where $T^{-1}$ is the frame operator for $\{g_n\}$. By Lemma \ref{lem:A1}, we know that $T \phi_n = \psi_n$ for all $n \in \N_0$. Again define $\{e_n\}$ and $\{h_n\}$ by Equation \eqref{defn:en} and Equation \eqref{defn:aux}, respectively. By Lemma \ref{lem:A2}, we know that $h_n = T^{\frac{1}{2}}g_n = T^{-\frac{1}{2}}\tilde{g}_n$ for all $n \in \N_0$. 

Since $\{g_n\}$ and $\{\tilde{g_n}\}$ are dual frames, we know that $\{T^{\frac{1}{2}}g_n\}$ and $\{T^{-\frac{1}{2}} \tilde{g}_n\}$ are also dual frames.
For~$x \in \mch$, observe that
\[x = \sum_{n=0}^{\infty} \la x, T^{-\frac{1}{2}} \tilde{g}_n \ra T^{\frac{1}{2}}g_n= \sum_{n=0}^{\infty} \la x, h_n \ra h_n\]
from which it follows that
\[\norm{x}^2 = \la x, x \ra = \sum_{n=0}^{\infty} \la x, h_n \ra \la h_n,x \ra = \sum_{n=0}^{\infty} \abs{\la x, h_n\ra}^2.\]
Therefore, $\{h_n\}$ is a Parseval frame, so $\{e_n\}$ is effective by \cite{Kwa-Myc}. Noting that $\phi_n = T^{-\frac{1}{2}}e_n$ and $\psi_n = T^{\frac{1}{2}}e_n$, we conclude that $\{(\phi_n, \psi_n)\}$ is a symmetric effective pair by Corollary \ref{cor:T pair to seq}.
\end{proof}

\begin{thm}\label{thm:equiv1}
Suppose that $\{\phi_n\}$ and $\{\psi_n\}$ are linearly dense sequences in $\mch$, whose respective auxiliary sequences are $\{g_n\}$ and $\{\tilde{g}_n\}$ as in Equations \eqref{defn:dual aux} and \eqref{defn: dual aux2}. Assume ${\la \phi_n, \psi_n \ra = 1}$ for all $n \in \N_0$ and suppose there exists a positive, invertible $T \in \mcb(\mch)$ such that $T \phi_n = \psi_n$ for all $n \in \N_0$. The following are then equivalent:
\begin{enumerate}
	\item[$(i)$] $V$  is a partial isometry, where $V$ is given by Equation \eqref{eq:N and V}.
	\item[$(ii)$] $\{g_n\}$ and $\{\tilde{g}_n\}$ are canonical dual frames. 
	\item[$(iii)$] $\{ (\phi_n, \psi_n) \}$ is a symmetric effective pair.
\end{enumerate}
Moreover, if any of these conditions hold, then $T^{-1}$ is the frame operator for $\{g_n\}$. 
\end{thm}
\begin{proof}
We will show

\centerline{$(i) \Rightarrow (iii) \Rightarrow (ii) \Rightarrow (i).$}

Suppose $(i)$ holds. It is immediate from Proposition \ref{prop:V implies sym ep} that $\{(\phi_n, \psi_n)\}$ is a symmetric effective pair.

Suppose $(iii)$ holds. Define the sequences $\{e_n\}$ and $\{h_n\}$ by Equations \eqref{defn:en} and \eqref{defn:aux}, respectively.  As $e_n = T^{\frac{1}{2}}\phi_n$ and $e_n = T^{-\frac{1}{2}}\psi_n$, we have that $\{e_n\}$ is an effective sequence by Corollary \ref{cor:T pair to seq}, implying that $\{h_n\}$ is a Parsevel frame by \cite{Kwa-Myc}. By Lemmas \ref{lem:A3} and \ref{lem:A2}, we know that $\tilde{g}_n = Tg_n$ and $h_n = T^\frac{1}{2}g_n$. 
 As $g_n = T^{-\frac{1}{2}}h_n$, $\tilde{g}_n = T^{\frac{1}{2}}h_n$, and $\{h_n\}$ is a Parseval frame, we know that $\{g_n\}$ and $\{\tilde{g}_n\}$ are dual frames. Moreover, since $\tilde{g}_n = Tg_n$ for invertible $T$, we conclude that $\{g_n\}$ and $\{\tilde{g}_n\}$ must be canonical dual frames.

Suppose $(ii)$ holds.
 Since $T$ is self-adjoint, it is straightforward to verify Equation \eqref{eq:phi psi sym}, so by Proposition \ref{prop:dual implies sym ep}, we infer that $\{(\phi_n, \psi_n)\}$ is a symmetric effective pair. Defining $\{e_n\}$ by Equation \eqref{defn:en} and applying Corollary \ref{cor:T pair to seq}, we see that $\{e_n\}$ is an effective sequence. Appealing to \cite[Theorem~1]{Hall-Szw}, we conclude that the associated matrix $U$, as defined by Equation \eqref{I+M}, is a partial isometry. By Equation \eqref{eq:en sym}, $U=V$, and we have the desired result.

Moreover, if $(ii)$ holds, we have $\tilde{g}_n = S^{-1}g_n$, where $S$ is the frame operator of $\{g_n\}$. By Lemma \ref{lem:A3}, we also know that $\tilde{g}_n = Tg_n$. Because $\{g_n\}$ and $\{\tilde{g}_n\}$ are canonical dual frames, $T$ must be unique and thus $S=T^{-1}$.
\end{proof}

For the remainder of this section, we will confine ourselves to finite-dimensional Hilbert spaces, where our characterization effort is aided by the existence of a positive, invertible $T \in \mcb(\mch)$ relating the sequences $\{\phi_n\}$ and $\{\psi_n\}$. We present necessary and sufficient conditions for the existence of such a $T$ and then use this result to present a partial characterization of effective pairs in finite dimensions. 

\begin{lem}\label{lem:finite dim A} Suppose $\{\phi_n\}$ and $\{\psi_n\}$ are linearly dense sequences in a finite dimensional Hilbert space $\mch_N$. Then there exists a positive, invertible  $T \in \mcb(\mch_N)$ such that $T \phi_n = \psi_n$ if and only if $\Theta_\psi \Theta_\phi^*$ is positive.
\end{lem}

\begin{proof}
Suppose that there exists a positive, invertible $T \in \mcb(\mch_N)$ such that $T\phi_n = \psi_n$. First, we show that $\Theta_{\psi}\Theta_{\phi}^*$ is self-adjoint. Let $\{\delta_n\}$ be the canonical orthonormal basis of $\ell^2(\mathbb{N}_0)$. Observe
\[\Theta_{\phi}\Theta_{\psi}^*\delta_n = \{\langle \psi_n,\phi_k\rangle\}_{k=0}^{\infty} 
= \{\langle T\phi_n,\phi_k\rangle\}_{k=0}^{\infty} 
= \{\langle \phi_n,T\phi_k\rangle\}_{k=0}^{\infty} 
= \{\langle \phi_n,\psi_k\rangle\}_{k=0}^{\infty} 
= \Theta_{\psi}\Theta_{\phi}^*\delta_n.\]
From this it immediately follows that $(\Theta_{\psi}\Theta_{\phi}^*)^* = \Theta_{\phi}\Theta_{\psi}^* = \Theta_{\psi}\Theta_{\phi}^*$. Next, we show that $\Theta_{\psi}\Theta_{\phi}^*$ is positive. Observe for any finite sequence $\{c_n\}$ that 

\begin{align*}
 \left\langle \Theta_{\phi}\Theta_{\psi}^*\sum_{j=0}^{\infty}c_j\delta_j,\sum_{k=0}^{\infty}c_k\delta_k\right\rangle 
&= \sum_{j=0}^{\infty}\sum_{k=0}^{\infty}c_j\overline{c_k}\langle \Theta_{\phi}\Theta_{\psi}^*\delta_j,\delta_k\rangle
= \sum_{j=0}^{\infty}\sum_{k=0}^{\infty}c_j\overline{c_k}\langle \psi_j,\phi_k\rangle \\
&= \left\langle \sum_{j=0}^{\infty}c_j\psi_j,\sum_{k=0}^{\infty}c_k\phi_k\right\rangle 
= \left\langle T\sum_{j=0}^{\infty}c_j\phi_j,\sum_{k=0}^{\infty}c_k\phi_k\right\rangle \geq 0.
\end{align*}
Therefore $\Theta_{\phi}\Theta_{\psi}^*$, and thus $\Theta_{\psi}\Theta_{\phi}^*$, is positive.

Conversely, suppose that $\Theta_{\psi}\Theta_{\phi}^*$ is positive. As $\mch_N$ is finite dimensional, there is some $M \in \N$ such that $\{\phi_n\}_{n=0}^{M-1}$ and $\{\psi_n\}_{n=0}^{M-1}$ both span $\mch_N$. It follows that $\{\phi_n\}_{n=0}^{M-1}$ and $\{\psi_n\}_{n=0}^{M-1}$ are frames for $\mch_N$, which ensures that the operators  
\[\theta_\phi: \mch_N \to \ell^2(\Z_M), \hspace{12pt} \theta_\phi x = \{\la x, \phi_n \ra\}_{n=0}^{M-1}\]
\[\theta_\phi^*: \ell^2(\Z_M) \to \mch_N, \hspace{12pt} \theta_\phi^* \{c_n\} = \sum_{n=0}^{M-1} c_n \phi_n\]
are well defined, 
and that $\theta_\phi^*$ and $\theta_\psi^*$ (defined analogously to $\theta_\phi^*$) are surjective. As $\theta_\psi\theta_\phi^*$ is positive by assumption, we have that $\theta_\psi\theta_\phi^*= \theta_\phi\theta_\psi^*$ and thus $ran{\theta_\phi} = ran{\theta_\psi}$. Let $B = ran\theta_\phi = ran\theta_\psi$. Then $\theta_\phi^*$ and $\theta_\psi^*$ are invertible when restricted to $B$ since $B = (ker\theta_\phi^*)^\perp = (ker\theta_\psi^*)^\perp$. Let $T : \mch_N \rightarrow \mch_N$ be given by $T = \theta_{\psi}^*\big|_B(\theta_{\phi}^*\big|_B)^{-1}$. 

Let $P$ be the orthogonal projection of $\ell^2(\Z_M)$ onto the closed subspace $B$. Then the operator $\hat{\theta}_{\phi} := P\theta_{\phi}$ from $\mch_N$ onto $B$ is invertible, and we may write $T = \theta_\psi^*\big|_B(\hat{\theta}_\phi \circ\theta_\phi^*\big|_B)^{-1} \hat{\theta}_\phi$. 

Let $\{\delta_n\}$ be the canonical basis for $\ell^2(\Z_M)$. Note that $\delta_n - P\delta_n \in B^{\perp} = \ker\theta_\phi^*$. Thus $$T\phi_n = \theta_\phi^*\delta_n = \theta_\phi^*P\delta_n + \theta_\phi^*(\delta_n - P\delta_n) = \theta_\phi^*\big|_B P\delta_n$$ and similarly $\psi_n = \theta_\psi^*\big|_BP\delta_n$. We then have
\[\phi_n = T\theta_\phi^*\big|_B P\delta_n = \theta_\psi^*\big|_B(\hat{\theta}_\phi \circ\theta_\phi^*\big|_B)^{-1}(\hat{\theta}_\phi\circ\theta_\phi^*\big|_B)P\delta_n = \theta_\psi^*\big|_B P\delta_n = \psi_n.\]
Note that, for such an operator $\hat{T}$ constructed on a larger collection $\hat{M} \geq M$, $\hat{T}$ must agree with $T$ on a spanning set and, hence, $\hat{T} = T$.

Lastly, for any $x \in \mathcal{H}_N$, there is some $\{c_n\} \in B$ such that $\theta_\phi^*(\{c_n\}) = x$. Then
\[\la T^*x, x \ra = \left\la (T^*\theta_{\phi}^*)(\{c_n\}) , \theta_{\phi}^*(\{c_n\}) \right\ra = \left\la \theta_{\phi}^*(\{c_n\}), \theta_{\psi}^*(\{c_n\})\right\ra = \left\la (\theta_{\psi}\theta_{\phi}^*)(\{c_n\}), \{c_n\} \right\ra \geq 0\]
from which we conclude that $T^*$, and thus $T$, is positive.
\end{proof}

\begin{cor}\label{cor:equiv2}
If $\{\phi_n\}$ and $\{\psi_n\}$ are linearly dense sequences in a finite dimensional Hilbert space $\mch_N$ such that $\la \phi_n, \psi_n \ra = 1$ for all $n \in \N_0$ and $\Theta_\psi \Theta_\phi^*$ is positive, then the following are equivalent:
\begin{enumerate}
\item[$(i)$] $V$ is a partial isometry.
\item[$(ii)$] $\{g_n\}$ and $\{\tilde{g}_n\}$ are canonical dual frames.
\item[$(iii)$] $\{(\phi_n, \psi_n)\}$ is a symmetric effective pair.
\end{enumerate}
\end{cor}
\begin{proof}
This is an immediate consequence of Lemma \ref{lem:finite dim A} and Theorem \ref{thm:equiv1}.
\end{proof}

\subsection{Almost effective sequences and the augmented dual Kaczmarz algorithm}

Although the Kaczmarz algorithm provides many computational advantages, the class of effective sequences is too rigid to tolerate general perturbation. Earlier in the paper, we addressed this rigidity issue by introducing effective pairs. In \cite{Cza-Tan}, Czaja and Tanis also sought to relax the concept of an effective sequence. Specifically, they defined a new class of sequences by calling $\{e_n\}$ in $\mch$ \textit{almost effective} if there exists some $0 \leq B < 1$ such that the sequence $\{x_n\}$ in Equation \eqref{defn:ka} satisfies
\begin{equation}\label{def:ae}
\lim_{n\rightarrow\infty}\|x_n - x\|^2 \leq B\|x\|^2 \quad \text{for all } x \in \mch.
\end{equation}

In \cite{Cza-Tan}, Czaja and Tanis proved that a sequence
is almost effective with bound $0 \leq (1-A)$ if and
only if its corresponding auxiliary sequence defined by \eqref{defn:aux} is a frame
with bounds $0 < A \leq 1$. This characterization provides
another succinct connection between the Kaczmarz algorithm
and frame theory.

Although almost effective sequences provide more flexibility, they are accompanied by nontrivial disadvantages. Recall the original impetus for our investigation into effective sequences---to reconstruct a vector given its inner products with some linearly dense sequence of unit vectors.  
While an effective sequence yields such a reconstruction directly via the Kaczmarz algorithm, an almost effective sequence does not necessarily retain this property. 
By combining the idea of an almost effective sequence with that of an effective pair, however, we are able to attain approximations based upon the desired inner products.  

Similar to Szwarc in \cite{Szw-all}, who showed that a certain type of Bessel sequence generates an effective sequence, we start with a lemma showing that canonical dual frames satisfying the appropriate orthogonality condition generate a symmetric effective pair. This will be an essential tool for our results involving almost effective sequences. 

\begin{lem} \label{lem: pair from gns}
Suppose that $\{g_n\}$ and $\{\tilde{g}_n\}$ are canonical dual frames in a Hilbert space $\mch$ such that
\[g_0 \perp \tilde{g}_n \quad \text{for all} \  n \in \N. \]
Then there exists a symmetric effective pair $\{(\phi_n, \psi_n)\}$ with auxiliary sequences $\{g_n\}$ and $\{\tilde{g}_n\}$, as in Equations \eqref{defn:dual aux} and \eqref{defn: dual aux2}.  
\end{lem}

\begin{proof}
As $\{g_n\}$ and $\{\tilde{g}_n\}$ are canonical dual frames, we have that $S^{-1}g_n = \tilde{g}_n$ where $S$ is the frame operator for $\{g_n\}$. Define a sequence $\{f_n\}$ by $f_n = S^{-\frac{1}{2}} g_n$. Observe that $\{f_n\}$ is a Parseval frame and that $f_0 \perp f_n$ for all $n \in \mathbb{N}$ as
\[\la f_0, f_n \ra = \la S^{-\frac{1}{2}}g_0, S^{-\frac{1}{2}} g_n \ra = \la g_0, S^{-1} g_n \ra = \la g_0, \tilde{g}_n \ra = 0. \]

From \cite[Theorem~1]{Szw-all}, we know $\{f_n\}$ is the auxiliary sequence for some effective sequence, say $\{b_n\}$, in $\mch$.  Define the sequences $\{\phi_n\}$ and $\{\psi_n\}$ by $\phi_n = S^{\frac{1}{2}} b_n$ and $\psi_n = S^{-\frac{1}{2}}b_n$. By Theorem \ref{thm:T to pair}, Lemma \ref{lem:A2}, and Lemma \ref{lem:A3}, it follows that $\{ (\phi_n, \psi_n)\}$ is a symmetric effective pair with auxiliary sequences $\{S^{\frac{1}{2}} f_n\} = \{g_n\}$ and $\{\tilde{g}_n\}$. 
\end{proof}

Now that we have a method for generating an effective pair corresponding to certain auxiliary sequences, we use it to produce an effective pair with the same auxiliary sequence $\{h_n\}$ as a given almost effective sequence $\{e_n\}$. 

\begin{thm}\label{thm:pair from aes}
Suppose that $\{e_n\}$ is an almost effective sequence in a Hilbert space $\mch$ with auxiliary sequence $\{h_n\}$. Then there exists a symmetric effective pair $\{(\phi_n, \psi_n)\}$ with auxiliary sequences $\{g_n\}$ and $\{\tilde{g}_n\}$, as in Equations \eqref{defn:dual aux} and \eqref{defn: dual aux2}, such that 
\begin{itemize}
    \item[$(i)$] $h_n = g_n$ for all $n \in \N_0$.
    \item[$(ii)$]$\{g_n\}$ and $\{\tilde{g}_n\}$ are canonical dual frames.
\end{itemize}
Moreover, $x$ can be reconstructed from $\{\langle x, h_n \rangle\}$ by 
\begin{align*}
    x &= \sum \limits_{n=0}^{\infty} \langle x, h_n \rangle \psi_n.
\end{align*}

\end{thm}
\begin{proof} 
As $\{e_n\}$ is almost effective, its auxiliary sequence $\{h_n\}$ is a frame with Bessel bound $0<B\leq 1$
\cite[Theorem~3.1]{Cza-Tan}. Since $\norm{h_0}^2 = \norm{e_0}^2 = 1$, it follows from the Bessel inequality that $\la h_0, h_n \ra = 0$ for all $n\in \N$.

Let $S$ be the frame operator of $\{h_n\}$. Define the canonical dual frames $\{g_n\}$ and $\{\tilde{g}_n\}$, where 
\[g_n = h_n, \qquad \tilde{g}_n = S^{-1}h_n.\]

As $h_0 \perp h_n$ for $n\in \N$, we infer that 
\[Sh_0 = \sum_{n=0}^\infty \la h_0, h_n \ra h_n = h_0.\] 
For $n \in \mathbb{N}$, we then have
\[\la g_0, \tilde{g}_n \ra = \la h_0, S^{-1}h_n \ra = \la S^{-1}h_0, h_n \ra = \la h_0, h_n \ra= 0\]
and
\[\la \tilde{g}_0, g_n \ra = \la S^{-1}h_0, h_n \ra=\la h_0, h_n \ra = 0.\]

By Lemma \ref{lem: pair from gns}, there are sequences $\{\phi_n\}$ and $\{\psi_n\}$ in $\mch$ such that $\{(\phi_n,\psi_n)\}$ is a symmetric effective pair with auxiliary sequences $\{g_n\}$ and $\{\tilde{g}_n\}$. 

Furthermore, as $h_n=g_n$, by the reconstruction in Equation \eqref{pair recon} we know that
\[x =\sum_{k=0}^\infty \langle x,g_k\rangle\psi_k =\sum_{k=0}^{\infty} \langle x,h_k\rangle\psi_k \quad \text{for all} \ x \in \mathcal{H}.\]
\end{proof}

We now have a sequence of approximations to $x$ generated by inner products with the auxiliary sequence of an almost effective sequence. In the following corollary, we introduce another variation on the Kaczmarz algorithm which will allow us to achieve reconstruction based upon the inner products with the almost effective sequence itself. 

\begin{cor}
Suppose that $\{e_n\}$ is an almost effective sequence in a Hilbert space $\mch$ with auxiliary sequence $\{h_n\}$. Let $\{\psi_n\}$ be as in the conclusion of Theorem \ref{thm:pair from aes}. For any $x \in \mch$, let $\{x_n\}$ be the sequence generated from $\{e_n\}$ as in Equation \eqref{defn:ka}. Furthermore, define the sequence $\{y_n\}$ by 
\begin{equation}\label{aug ka}
\begin{aligned}
y_0 &=\langle x,e_0\rangle\psi_0, \\
y_n &= y_{n-1} + \langle x - x_{n-1},e_n\ra \psi_n, \quad n \geq 1.
\end{aligned}
\end{equation}
Then, $\lim_{n\to \infty}||y_n-x||=0$.
\end{cor}

We call the new formulation in \eqref{aug ka} the \textit{augmented dual Kaczmarz algorithm}. Note that, as $\{e_n\}$ is merely almost effective, we do not make any assumptions about the convergence of $\{x_n\}$. Indeed, even if $\lim_{n \to \infty} x_n$ exists, it need not be equal to $x$. However, we use the sequence $\{x_n\}$ as a state variable to gain the sequence of approximations $\{y_n\}$ generated in \eqref{aug ka}. 

\begin{proof}
From Theorem \ref{thm:pair from aes}, we know that 
\[x =\sum_{k=0}^{\infty} \langle x,h_k\rangle\psi_k \quad \text{for all} \ x \in \mch,\]
so it suffices to show that 
\begin{equation}\label{yn sum}
y_n  = \sum_{k=0}^n \langle x,h_k\rangle\psi_k \quad \text{for all} \ x \in \mch, n \in \N_0.
\end{equation}

This is clear for $n=0$ as $e_0=h_0$. Assume inductively that the claim holds for $0 \leq k < n$ and note that
\[\la x - x_{n-1}, e_n \ra = \la x, e_n \ra - \left\la \sum_{k=0}^{n-1} \la x, h_k \ra e_k, e_n \right\ra =  \left\la x, e_n - \sum_{k=0}^{n-1} \la e_n, e_k \ra h_k \right\ra  = \la x, h_n \ra. \]
We then have
\[ y_n =  y_{n-1} + \langle x - x_{n-1},e_n\rangle\psi_n  = \sum_{k=0}^{n-1} \langle x,h_k\rangle\psi_k + \langle x - x_{n-1},e_n\rangle\psi_n = \sum_{k=0}^{n} \langle x,h_k\rangle\psi_k \]
and thus Equation \eqref{yn sum} holds.
\end{proof}

\section{Examples}

In this section, we list examples which illuminate some of the interesting characteristics of effective pairs.
\begin{obs}
Suppose that $\{\phi_n\}$ and $\{\psi_n\}$ are effective. It is not necessarily true that $\{(\phi_n, \psi_n)\}$ is an effective pair.
\end{obs}
 The most straightforward example would be to consider an orthonormal basis $\{\phi_n\}$ and take ${\psi_n = -\phi_n}$. Clearly $\{\phi_n\}$ and $\{\psi_n\}$ are effective sequences. However, the corresponding dual Kaczmarz algorithm applied to $x$ reproduces $-x$, i.e. $\|x_n - (-x)\| \rightarrow 0$, so that $\{(\phi_n,\psi_n)\}$ is not an effective pair. This is immediate from
\[x_n = \sum_{k=0}^n \langle -x,\phi_k\rangle\phi_k\]
which follows from Equation (\ref{xn approx}).

\begin{obs}\label{obs:symmetry}
There are effective pairs $\{(\phi_n,\psi_n)\}$ satisfying $\langle \phi_n,\psi_n\rangle = 1$ for all $n \in\N_0$ which are not symmetric effective pairs.
\end{obs}
In $\mathbb{R}^2$, consider the periodic sequences $\{\phi_n\}$ and $\{\psi_n\}$ with
\[\left[\begin{array}{c|c|c|c|c|c}
\phi_0 & \phi_1 & \phi_2 & \psi_0 & \psi_1 & \psi_2
\end{array}\right] = \left[\begin{array}{c|c|c|c|c|c}
1 & 1 & 0.5 & 1 & 1 & 1.5 \\
-1 & 1 & -0.5 & 0 & 0 & -0.5
\end{array}
\right]\]
and $\phi_n = \phi_{(n\mod 3)}$, $\psi_n = \psi_{(n\mod 3)}$. Consider the error sequence $\{\varepsilon_n \,|\, \varepsilon_n = x - x_n\}$ corresponding to the dual Kaczmarz algorithm for $\{(\phi_n,\psi_n)\}$ associated to $x$ where
\begin{align*}
\varepsilon_0 &= x - \langle x,\phi_0\rangle\psi_0 \\
\varepsilon_n &= \varepsilon_{n-1} - \langle \varepsilon_{n-1},\phi_n\rangle\psi_n, \quad n \geq 1.
\end{align*}
Then $\{(\phi_n,\psi_n)\}$ is an effective pair if and only if $\varepsilon_n \rightarrow 0$. It is simple to show by induction that the error sequence $\{\varepsilon_n\}$ associated to $x = (a,b)$ satisfies
\[\varepsilon_{3k} = \dfrac{b}{2^k}\begin{pmatrix}
1 \\
1
\end{pmatrix}, \quad \varepsilon_{3k+1} = \dfrac{b}{2^k}\begin{pmatrix}
-1 \\
1
\end{pmatrix}, \quad \varepsilon_{3k+2} = \dfrac{b}{2^{k+1}}\begin{pmatrix}
1 \\
1
\end{pmatrix}, \quad k \geq 0.\]
Therefore, $\{(\phi_n,\psi_n)\}$ is an effective pair. However, $\{(\psi_n,\phi_n)\}$ is not an effective pair for the following reason: Let $\{\varepsilon_n\}$ be the error sequence associated to $x = (0,4)$. Then, by induction, we find
\[\varepsilon_{3k} = \begin{pmatrix}
0 \\
4
\end{pmatrix},\quad \varepsilon_{3k+1} = \begin{pmatrix}
0 \\
4
\end{pmatrix},\quad \varepsilon_{3k+2} = \begin{pmatrix}
1 \\
3
\end{pmatrix},\quad k \geq 0. \]
The sequence $\{\varepsilon_n\}$ fails to converge.

\begin{obs}
There are symmetric effective pairs $\{( \phi_n,\psi_n)\}$ for which the mixed 
Grammian operator $\Theta_{\phi}\Theta_{\psi}^*$ is not positive. Furthermore, there are symmetric effective pairs which are not related by an invertible operator, i.e. there does not exist an invertible $T \in \mcb(\mch)$ such that $T\phi_n = \psi_n$ for all $n$.
\end{obs}

In $\mathbb{R}^2$, consider the periodic sequences $\{\phi_n\}$ and $\{\psi_n\}$ with
\[\left[\begin{array}{c|c|c|c|c|c}
\phi_0 & \phi_1 & \phi_2 & \psi_0 & \psi_1 & \psi_2
\end{array}\right] = \left[\begin{array}{c|c|c|c|c|c}
1 & 1 & 0 & 1 & 1 & 0 \\
0 & 1 & 1 & 0 & 0 & 1
\end{array}
\right]\]
and $\phi_n = \phi_{(n\mod 3)}$, $\psi_n = \psi_{(n\mod 3)}$. As in the previous example, consider the error sequence $\{\varepsilon_n\}$ for the pair $\{(\phi_n,\psi_n)\}$. Since $\varepsilon_2$ is the projection of $\varepsilon_1$ onto the orthogonal complement of $\phi_2$ and $\varepsilon_3$ is the projection of $\varepsilon_2$ onto the orthogonal complement of $\phi_0$ and $\{\phi_0,\phi_2\}$ form an orthonormal basis, it follows that $\varepsilon_k = (0,0)$ for $k \geq 3$, so $\{(\phi_n,\psi_n)\}$ is an effective pair. Likewise, by the same argument, we observe that $\{(\psi_n,\phi_n)\}$ is an effective pair. The matrix $\Theta_{\phi}\Theta_{\psi}^*$ is not positive since its $3 \times 3$ principle submatrix,
\[\begin{pmatrix}
1 & 1 & 0 \\
1 & 1 & 1 \\
0 & 0 & 1
\end{pmatrix},\]
is not positive. Note that an invertible $T \in \mcb(\mch)$ can not possibly map $\phi_n$ to $\psi_n$ for all $n \in \N_0$ since $\psi_0 = \psi_1$ yet $\phi_0 \neq \phi_1$.

\begin{obs}\label{obs:not dual frames}
There are symmetric effective pairs $\{(\phi_n,\psi_n)\}$ where neither $\{\phi_n\}$ nor $\{\psi_n\}$ is effective. Moreover, there are symmetric effective pairs for which their auxiliary sequences do not form (dual) frames.
\end{obs}

Let $\{\phi_n\}$ be a Schauder basis which is not a Riesz basis, and let $\{\psi_n\}$ be its biorthogonal dual basis. We then have the reconstruction property
$$x = \sum_{n=0}^{\infty}\langle x,\phi_n\rangle \psi_n = \sum_{n=0}^{\infty}\langle x,\psi_n\rangle\phi_n.$$
Since the auxiliary sequence of $\{(\phi_n,\psi_n)\}$ is $\{\psi_n\}$ and the auxiliary sequence of $\{(\psi_n,\phi_n)\}$ is $\{\phi_n\}$, it follows that $\{(\phi_n,\psi_n)\}$ is a symmetric effective pair where the auxiliary sequences are not (dual) frames. 

Moreover, if $\{\phi_n\} \subset X$ and $\{\psi_n\} \subset X^*$ for a Banach space $X$, then these sequences form an effective pair as in definition \eqref{defn:dual ka}, while also satisfying Kwapie\'n and Mycielski's definition of an effective sequence from \eqref{defn:ka in Banach}.  

\section*{Acknowledgements}

 Anna Aboud and Eric Weber were supported in part by the National Science Foundation and the National Geospatial-Intelligence Agency under NSF award \#1832054.



\end{document}